\documentclass[10pt,leqno]{article}
\usepackage{amsmath, amssymb, amsthm}
   \swapnumbers
   \makeatletter
     \def\swappedhead@plain#1#2#3{%
       \thmnumber{(\textup{#2})}
       \thmname{\@ifnotempty{#2}{~}\textup{#1}}
       \thmnote{ {\textup{(#3)}}}}
     \let\swappedhead\swappedhead@plain
   \makeatother
   \theoremstyle{plain}
   \newtheorem{theo}[equation]{Theorem}
   \newtheorem{prop}[equation]{Proposition}

   \theoremstyle{definition}

   \newtheorem{rems}[equation]{Remarks}
   \newtheorem{dual}[equation]{Prequantum dual}
   \newtheorem{boxp}[equation]{Prequantum product}
   \newtheorem{redu}[equation]{Prequantum reduction}
   \numberwithin{equation}{section}
\usepackage[T1]{fontenc}               
\usepackage[utf8]{inputenc}            
   
   \let\dotlessi\i 
   
   \let\polishl\l
   \let\norwegiano\o
   
   \let\russianbreve\u
\usepackage[dvipsnames]{xcolor}        
\usepackage{tikz-cd}
\usepackage{Z-fonts}
\usepackage[final,%
            colorlinks=true,%
            linkcolor=RedViolet,%
            citecolor=RedViolet,
            urlcolor=RoyalBlue]{hyperref}   

  \DeclareMathSymbol{A}{\mathalpha}{operators}{`A}
  \DeclareMathSymbol{B}{\mathalpha}{operators}{`B}
  \DeclareMathSymbol{C}{\mathalpha}{operators}{`C}
  \DeclareMathSymbol{D}{\mathalpha}{operators}{`D}
  \DeclareMathSymbol{E}{\mathalpha}{operators}{`E}
  \DeclareMathSymbol{F}{\mathalpha}{operators}{`F}
  \DeclareMathSymbol{G}{\mathalpha}{operators}{`G}
  \DeclareMathSymbol{H}{\mathalpha}{operators}{`H}
  \DeclareMathSymbol{I}{\mathalpha}{operators}{`I}
  \DeclareMathSymbol{J}{\mathalpha}{operators}{`J}
  \DeclareMathSymbol{K}{\mathalpha}{operators}{`K}
  \DeclareMathSymbol{L}{\mathalpha}{operators}{`L}
  \DeclareMathSymbol{M}{\mathalpha}{operators}{`M}
  \DeclareMathSymbol{N}{\mathalpha}{operators}{`N}
  \DeclareMathSymbol{O}{\mathalpha}{operators}{`O}
  \DeclareMathSymbol{P}{\mathalpha}{operators}{`P}
  \DeclareMathSymbol{Q}{\mathalpha}{operators}{`Q}
  \DeclareMathSymbol{R}{\mathalpha}{operators}{`R}
  \DeclareMathSymbol{S}{\mathalpha}{operators}{`S}
  \DeclareMathSymbol{T}{\mathalpha}{operators}{`T}
  \DeclareMathSymbol{U}{\mathalpha}{operators}{`U}
  \DeclareMathSymbol{V}{\mathalpha}{operators}{`V}
  \DeclareMathSymbol{W}{\mathalpha}{operators}{`W}
  \DeclareMathSymbol{X}{\mathalpha}{operators}{`X}
  \DeclareMathSymbol{Y}{\mathalpha}{operators}{`Y}
  \DeclareMathSymbol{Z}{\mathalpha}{operators}{`Z}

  \newcommand\CC{{\mathbf C}}        
  \newcommand\RR{{\mathbf R}}        
  \newcommand\TT{{\mathbf T}}        
  \newcommand\ZZ{{\mathbf Z}}        
  
  \newcommand\Lie{\mathfrak}
  \newcommand\LG{{\Lie{g}}}
  \newcommand\LH{{\Lie{h}}}
  \newcommand\LK{{\Lie{k}}}

\let\Im\relax
\let\Re\relax
\DeclareMathOperator\ann{ann}                                
\DeclareMathOperator\Hom{Hom}
\DeclareMathOperator\Im{Im}
\DeclareMathOperator\Ker{Ker}
\DeclareMathOperator\Re{Re}
\DeclareMathOperator\Res{Res}
  \newcommand\e[1]{{\mathrm e^{\hspace{.06em}#1}}}           
\renewcommand\i{{\mathrm i}}                                 
  \newcommand\IND[3]{\smash{\operatorname{Ind}_{#1}^{#2}#3}}
  \newcommand\inv{^{-1}}                                     
  \newcommand\subg{_{\,\smash{|}\LG}}
  \newcommand\subh{_{\,\smash{|}\LH}}
  \newcommand\subhprime{_{\,\smash{|}\LH'}}
  \newcommand\subk{_{\,\smash{|}\LK}}

  \newcommand\<{\langle}                                     
\renewcommand\>{\rangle}                                     

\usepackage{etoolbox}
\makeatletter
\newcommand*{\@linkedbibitem}[1]{%
  \def\this@biblink{#1}%
  \bibitem}
\newcommand*{\linkedbibitem}{\hyper@normalise\@linkedbibitem}
\renewcommand*{\@BIBLABEL}[1]{%
  \ifdefvoid\this@biblink
    {[#1]}
    {[\expandafter\href\expandafter{\this@biblink}%
       {#1}]}}
\makeatother

\title{Symplectic Induction, Prequantum Induction, and Prequantum Multiplicities}

\author{
Tudor S. Ratiu\footnote{
School of Mathematical Sciences and Ministry of Education Key 
Lab on Scientific and Engineering Computing, Shanghai Jiao Tong University, Shanghai 200240, China and Section de Math\'ematiques, Universit\'e de 
Gen\`eve and Ecole Polytechnique F\'ed\'erale de Lausanne, Switzerland.
\texttt{ratiu@sjtu.edu.cn, tudor.ratiu@epfl.ch}}
\ and 
Fran\c{c}ois Ziegler\footnote{
Department of Mathematical Sciences, Georgia Southern University, Statesboro, GA 30460-8093, USA. \texttt{fziegler@georgiasouthern.edu}\newline{\color{black}\quad} AMS subject classification (2020) 53D20, 53D50, 22D30.}
}

\date{March 2, 2021}

\begin{document}

\maketitle

\begin{abstract}
Frobenius reciprocity asserts that induction from a subgroup and restriction to it are adjoint functors in categories of unitary $G$-modules. In the 1980s, Guillemin and Sternberg established a parallel property of Hamiltonian $G$-spaces, which (as we show) unfortunately fails to mirror the situation where more than one $G$-module ``quantizes'' a given Hamiltonian $G$-space. This paper offers evidence that the situation is remedied by working in the category of \emph{prequantum} $G$-spaces, where this ambiguity disappears; there, we define induction and multiplicity spaces, and establish Frobenius reciprocity as well as the ``induction in stages'' property.
\end{abstract}

\section*{Introduction}

Beyond the mere parametrization of irreducible unitary representations by 
coadjoint orbits originating in the work of Borel-Weil and Kirillov 
\cite{Serre:1954, Kirillov:1962}, there exists a certain well-known 
\emph{parallelism} between representation theory and the symplectic theory of 
Hamiltonian $G$-spaces. To capture it with precision, papers like 
\cite{Kazhdan:1978, Weinstein:1978, Guillemin:1982, Guillemin:1983} introduced 
purely symplectic constructions meant to mirror operations such as 
$\mathit{Ind_H^G}$ (inducing a representation from a subgroup) or $\mathit{Hom_G}$ 
(forming the space of intertwining operators between two representations). In 
that setting, one of course expects basic properties like \emph{induction in 
stages} or \emph{Frobenius reciprocity} to hold in symplectic geometry. A 
first goal of this paper is to spell out their proofs (\S2, \S3), fulfilling 
promises made in \cite[p.\,9]{Ziegler:1996} and \cite[p.\,105]{Marsden:2007}. 

A second goal of the paper is to point out that, while these constructions fit 
their purpose when the correspondence from representations to coadjoint orbits 
is one-to-one, as in the Borel-Weil theory for compact groups or the Kirillov-Bernat theory for exponential groups \cite{Kirillov:1962, Fujiwara:2015}, they fall short when it is many-to-one, as in the Auslander-Kostant theory for solvable groups 
\cite{Auslander:1971a, Shchepochkina:1994}. To remedy this, we propose new versions of both 
constructions in the category of \emph{prequantum $G$-spaces} (\S5, \S6) and 
establish the stages and Frobenius properties in that setting (\S7, \S8). 
Finally, we illustrate the need for our prequantized versions by what we 
believe is the simplest example (\S4, \S9).

\subsubsection*{Notation and conventions}

We use a concise notation for the translation of tangent and cotangent vectors 
to a Lie group: for fixed $g,q\in G$,
\begin{equation}
	\gathered T_qG             \\\vspace{-3pt}       v       \endgathered
	\gathered  \,\to\,        \\\vspace{-3pt}   \,\mapsto\,  \endgathered
	\gathered T_{gq}G         \\\vspace{-3pt}      gv,       \endgathered
	\quad\qquad\text{resp.}\qquad\quad
	\gathered T^*_qG         \\\vspace{-3pt}       p         \endgathered
	\gathered  \,\to\,        \\\vspace{-3pt}  \,\mapsto\,   \endgathered
	\gathered T^*_{gq}G     \\\vspace{-3pt}      gp          \endgathered
\end{equation}
will denote the derivative of $q\mapsto gq$, respectively its contragredient, i.e., 
$\<gp,v\>=\<p,g\inv v\>$. Likewise, we define $vg$ and $pg$ with 
$\<pg,v\>=\<p,vg\inv\>$. 

By a \emph{Hamiltonian $G$-space} we mean the triple $(X, \omega, \Phi)$ of a 
manifold $X$ on which $G$ acts, a $G$-invariant symplectic form $\omega$ on it, 
and a $G$-equivariant momentum map $\Phi:X\to\LG^*$. We identify spaces $X_1$, 
$X_2$ which are \emph{isomorphic}, i.e., related by a $G$-equivariant 
diffeomorphism which transforms $\omega_1$ into $\omega_2$ and $\Phi_1$ into 
$\Phi_2$. If several are in play, we also use subscripts like $\omega_X$, 
$\Phi_X$. We recall two cardinal properties of the momentum map:
\begin{equation}
   \label{cardinal}
   \text{(a) }\Ker(D\Phi(x)) = \LG(x)^\omega
   \qquad\quad
   \text{(b) }\Im (D\Phi(x)) = \ann(\LG_x).
\end{equation}
The first is the orthogonal relative to $\omega$ of the tangent space $\LG(x)$ 
to the orbit $G(x)$, $ x \in  X$; the second is the annihilator in $\LG^*$ of the stabilizer 
Lie subalgebra $\LG_x\subset\LG$.

\section{Symplectic Induction}
Given a closed subgroup $H\subset G$ and a Hamiltonian $H$-space 
$(Y,\omega_Y, \Psi)$, \cite{Kazhdan:1978} constructs an \emph{induced} 
Hamiltonian $G$-space as follows. Let $\varpi$ denote the canonical 1-form on $T^*G$ given by $\varpi(\delta p) 
= \<p,\delta q\>$, where $\delta p \in T_p(T^*G)$, $\delta q =
\pi_*(\delta p) \in T_{\pi(p)}G$, and $\pi :T^*G \rightarrow  G$ is the 
canonical projection. Endow $N:=T^*G\times Y$ with the symplectic form 
$\omega:=d\varpi + \omega_Y$ and the $G\times H$-action $(g,h)(p,y)=
(gph\inv, h(y))$, where $h(y)$
denotes the action of $h \in  H$ on $y \in  Y$. This action 
admits the equivariant momentum map 
$\phi\times\psi:N\to\LG^*\times\LH^*$,
\begin{equation}
   \label{phi_and_psi}
   \left\{
   \begin{array}{lll}
      \rlap{$\phi$}\phantom{\psi}(p,y) & \!\!= & \!\!pq\inv\\[.5ex]
      \psi(p,y) & \!\!= & \!\!\Psi(y)- q\inv p\subh
   \end{array}
   \right.
   \rlap{\qquad$(p\in T^*_qG)$.}
\end{equation}
The \textit{induced manifold} is, by definition, the Marsden-Weinstein reduced 
space of $N$ at $0\in\LH^*$, i.e.
\begin{equation}
   \label{induced_manifold}
   \IND HGY:=N/\!\!/H = \psi\inv(0)/H.
\end{equation}
In more detail: the action of $H$ is free and proper (because it is free and 
proper on the factor $T^*G$, where it is the right action of $H$ regarded as a 
subgroup of the group $T^*G$ \cite[\S III.1.6]{Bourbaki:1972}); so $\psi$ is a submersion 
(\ref{cardinal}b), $\psi\inv(0)$ is a submanifold, and \eqref{induced_manifold} 
is a manifold; moreover $\omega_{|\psi\inv(0)}$ degenerates exactly along the 
$H$-orbits (\ref{cardinal}a), so it is the pull-back of a 
uniquely defined symplectic form, $\smash{\omega_{N/\!\!/H}}$, on the quotient. 
Furthermore, the $G$-action commutes with the $H$-action and preserves 
$\psi\inv(0)$, and its momentum map $\phi$ is constant on $H$-orbits. Passing to 
the quotient, we obtain the required $G$-action on $\IND HGY$ and momentum 
map $\smash{\Phi_{N/\!\!/H}}:\IND HGY\to\LG^*$. Note that since $\psi$ is a 
submersion and $H$ acts freely, \eqref{induced_manifold} has dimension equal to $\dim (N) - 2\dim (H)$, i.e.
\begin{equation}
   \dim(\IND HGY) = 2\dim(G/H) + \dim(Y).
\end{equation}

\section{Symplectic Induction in Stages}

\begin{theo}[Stages]
   \label{symplectic_stages}
   If $H \subset K \subset G$ are closed subgroups
   of the Lie group $G$\textup, 
   then
   \begin{equation*}
      \IND KG{\IND HKY} = \IND HGY.   
   \end{equation*}
\end{theo}

\begin{proof}
Let $(N, \omega, \phi\times \psi)$ be as in \S1 and consider 
$M=T^*G\times T^*K\times Y$ with $2$-form
$\omega_M = d\varpi_{T^*G}+d\varpi_{T^*K}+\omega_Y$
and $G\times K\times H$-action
   \begin{equation}
      (g,k,h)(p,\bar p,y)=(gp k\inv,k\bar p h\inv,h(y)).
   \end{equation}
This admits the equivariant momentum map 
$\varphi\times\bar\phi\times\bar\psi:M\to\LG^*\times\LK^*\times\LH^*$:
\begin{equation}
\left\{
\begin{array}{lll}
\varphi(p,\bar p,y) & \!\!= & \!\!pq\inv\\[.5ex]
\bar\phi(p,\bar p,y) & \!\!= & \!\!\bar p\bar q\inv-q\inv p\subk\\[.5ex]
\bar\psi(p,\bar p,y) & \!\!= & \!\!\Psi(y)- \bar q\inv\bar p\subh
\end{array}
\right.
\end{equation}
for $(p,\bar p)\in T^*_q G\times T^*_{\bar q} K$. Define $r:M\to N$ by 
$r(p,\bar p, y)=(p \bar q,y)$ and consider the commutative diagram in Fig.~1, 
where we have written $j_1,j_2,j_3$ 
   \begin{figure}[t] 
      \centering
      \begin{tikzcd}[sep=small, arrows={line width=1.2*rule_thickness}]
         M
         \ar[rrr,"r"]
         &&&
         N
         \\
         &&
         \bar\psi\inv(0)
         \ar[ull,hook',swap,"j_1" inner sep=0.2ex]
         \ar[dd,"\pi_1"]
         \\
         &&&
         (\bar\phi\times\bar\psi)\inv(0)
         \ar[ul,hook',swap,"j" inner sep=0.25ex]
         \ar[rrrrrrr,"s",dashed]
         \ar[dd,"\pi"]
         &&&&&&&
         \psi\inv(0)
         \ar[uulllllll,hook',swap,"j_3" inner sep=0.25ex]
         \ar[dddddd,"\pi_3"]
         \\
         &&
         M/\!\!/H
         \\
         &&&
         \bar\Phi_{M/\!\!/H}\inv(0)
         \ar[ul,hook',swap,"j_2" inner sep=0.25ex]
         \ar[dddd,"\pi_2"]
         \\
         \\
         \\
         \\
         &&&
         (M/\!\!/H)/\!\!/K
         \ar[rrrrrrr,"t",dashed]
         &&&&&&&
         N/\!\!/H
      \end{tikzcd}
      \caption{Construction of the isomorphism $t$.}
   \end{figure}
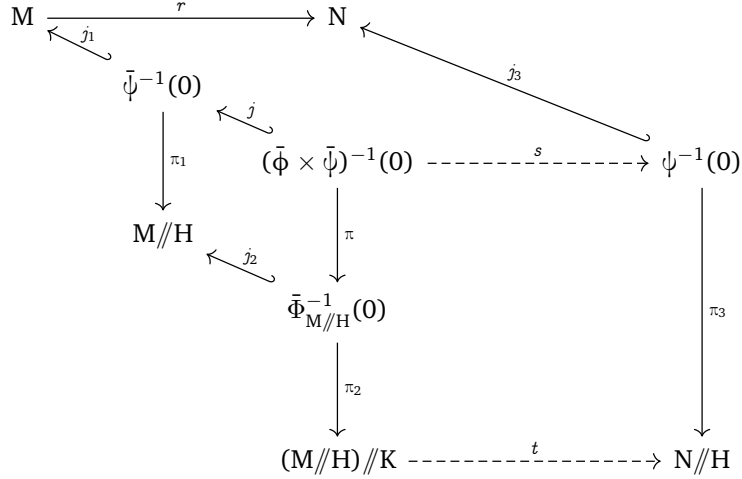
and $\pi_1,\pi_2,\pi_3$ for the inclusion and projection maps involved in 
constructing the reduced spaces $M/\!\!/H=T^*G\times\IND HKY$,  
$(M/\!\!/H)/\!\!/K=\IND KG{\IND HKY}$, and $N/\!\!/H=\IND HGY$; also $j,\pi$ 
are the obvious inclusion and restriction, and $\smash{\bar\Phi_{M/\!\!/H}}$ is 
the momentum map for the residual $K$-action on $M/\!\!/H$. The map 
$r\circ j_1 \circ j$ satisfies
\begin{equation}
  \begin{aligned}
	\psi((r\circ j_1 \circ j)(p,\bar p,y))
	  &=\psi(p\bar q,y)\\
	  &=\Psi(y)-(q\bar q)\inv p\bar q\subh\\
	  &=\Psi(y)-\bar q\inv (q\inv p)\subk\bar q\subh\\
	  &=\Psi(y)- \bar q\inv\bar p\subh &&\text{since }\bar\phi(p,\bar p,y)=0\\
	  &=0 &&\text{since }\bar\psi(p,\bar p,y)=0.
      \end{aligned}
   \end{equation}
\noindent
So $r\circ j_1\circ j$ takes values in $\psi\inv(0)$, i.e., there is a map 
$s$ as indicated in Fig.~1. Moreover, $s$ is onto since one 
verifies that $(p,y)\mapsto(p,(q\inv p)\subk,y)$ provides a right inverse.
The map $s$ is equivariant relative to the 
$G\times K \times H$-action on $(\bar{\phi} \times \bar{\psi})^{-1}(0)$
and the $G\times H$-action on $\psi^{-1}(0)$:
   \begin{equation}
      \begin{aligned}
	      s((g,k,h)(p,\bar p,y))
	      &=r(gp k\inv,k\bar p h\inv,h(y))\\
	      &=\smash{(gp\bar q h\inv, h(y))}\\
	      &=(g,h)(p\bar q,y)\\
	      &=(g,h)(s(p,\bar p,y)).
      \end{aligned}
   \end{equation}
Hence $s$ descends to a $G$-equivariant surjection $t$ as indicated in Fig.~1.
Furthermore, one checks without trouble that the fibers of $s$ are precisely 
the $K$-orbits in its domain. As $\pi_2\circ \pi$ collapses these orbits to 
points, it follows that $t$ is bijective, hence a diffeomorphism by 
\cite[{}5.9.6]{Bourbaki:1967}. The relation $\varphi=r^*\phi$ implies that 
$t$ relates the momentum maps for $G$: $\smash{\Phi_{(M/\!\!/H)/\!\!/K}}
=\smash{t^*\Phi_{N/\!\!/H}}$, so there only remains to see that 
$\smash{\omega_{(M/\!\!/H)/\!\!/K}}=\smash{t^*\omega_{N/\!\!/H}}$. To this end 
we compute
   \begin{equation}
      \begin{aligned}
	      (s^*j_3^*\varpi_{T^*G})(\delta p,\delta\bar p,\delta y)
	      &=\varpi_{T^*G}(\delta[p\bar q])\\
	      &=\<p\bar q,\delta[q\bar q]\>\\
	      &=\<p,\delta q\>+\smash{\<q\inv p,[\delta\bar q]\bar q\inv\>}
&&\text{since }\delta[q\bar q]=[\delta q]\bar q+q[\delta\bar q]\\
	      &=\<p,\delta q\>+\<\bar p,\delta\bar q\>
&&\text{since }\bar\phi(p,\bar p,y)=0\\
	      &=\varpi_{T^*G}(\delta p) + \varpi_{T^*K}(\delta\bar p).\phantom{\inv}
      \end{aligned}
   \end{equation}
\noindent
Taking exterior derivatives and adding $\omega_Y$ we obtain $s^*j_3^*\omega_{N}
=j^*j_1^*\omega_M$ or, equivalently (by commutativity of the diagram and 
definition of the reduced $2$-forms), 
$\smash{\pi^*\pi_2^*\,t^*\omega_{N/\!\!/H}}=
\smash{\pi^*\pi_2^*\,\omega_{(M/\!\!/H)/\!\!/K}}$. 
Since $\pi_2\circ \pi$ is a submersion, we are done. 
\end{proof}

\section{Symplectic Frobenius Reciprocity}

It is quite rare for an induced Hamiltonian $G$-space to be homogeneous 
or \emph{a fortiori} a coadjoint orbit (by which we mean that its momentum 
map is $1$-$1$ onto an orbit). In fact we have the following, where 
$\smash{\Phi_{N/\!\!/H}}$ is the momentum map for the induced space 
\eqref{induced_manifold}.

\begin{prop}
\label{pre_frobenius}
Let $(Y,\omega_Y,\Psi)$ be a Hamiltonian $H$-space.
\begin{enumerate}
   \item[\upshape(a)] A coadjoint orbit $\mathcal O$ of $G$ intersects 
   $\Im(\smash{\Phi_{N/\!\!/H}})\Leftrightarrow\smash{\mathcal O}\subh:=
   \{m\subh : m \in \mathcal{O}\}$ 
   intersects $\Im(\Psi)$.
   \item[\upshape(b)] If $\operatorname{Ind}_H^GY$ is homogeneous under $G$\textup, 
   then $Y$ is homogeneous under $H$.
   \item[\upshape(c)] If $\operatorname{Ind}_H^GY$ is a coadjoint 
   orbit of $G$\textup, then $Y$ is a coadjoint orbit of $H$.
   \end{enumerate}
\end{prop}

\begin{proof}
   (a): This re-expresses $\Im(\smash{\Phi_{N/\!\!/H}})=\phi(\psi\inv(0))$ 
\eqref{phi_and_psi}.
   (b): Assume $G$ is transitive on $\IND HGY$ and let $y_1, y_2\in Y$. Pick 
$m_i\in\LG^*$ such that $\Psi(y_i)=\smash{m_{i|\LH}}$. Then the $H$-orbits 
$x_i=H(m_i,y_i)$ are points in \eqref{induced_manifold}. So transitivity says 
that $x_1=g(x_2)$, i.e.
   \begin{equation}
      \label{transitivity}
	   (m_1,y_1)=(gm_2h\inv,h(y_2))
	   \quad\text{for some }
	   h\in H.
   \end{equation}
   In particular $y_1=h(y_2)$, as claimed.
   (c): Assume further that $\smash{\Phi_{N/\!\!/H}}$ is injective and 
suppose $\Psi(y_1)=\Psi(y_2)$. Then we can pick $m_1=m_2$ above. Since 
$\smash{\Phi_{N/\!\!/H}}(x_i)=m_i$ it follows, by injectivity, that $x_1=x_2$, 
i.e., we have \eqref{transitivity} with $g=e$. But then $h=e$ and hence 
$y_1=y_2$, as claimed.
\end{proof}

If $Y$ \emph{is} a coadjoint orbit, (\ref{pre_frobenius}a) says that  
$\IND HGY$ ``involves'' just those orbits $\mathcal O$ whose projection in 
$\LH^*$ contains $Y$. Guillemin and Sternberg \cite[\S6]{Guillemin:1982} 
proposed to measure the ``multiplicity'' of this involvement by the (possibly 
empty) space $\Hom_G(\mathcal O,\IND HGY)$, where we write suggestively
\begin{equation}
   \label{symplectic_hom}
   \Hom_G(X_1,X_2):=(X_1^-\times X_2^{\vphantom-})/\!\!/G,
\end{equation}
i.e., the Marsden-Weinstein reduction of $X_1^-\times X_2^{\vphantom-}$ at 
$0\in\LG^*$; here $X_1^-$ is the Hamiltonian $G$-space $X_1$ with its $2$-form 
and momentum map replaced by their negatives, and we regard \eqref{symplectic_hom} simply as a set. Then (\ref{pre_frobenius}a) can 
be refined by the following analog of Frobenius's theorem 
\cite[{}III.6.2]{Brocker:1985}, already found in 
\cite[Thm 2.2]{Guillemin:1983} when both $X$ and $Y$ are coadjoint orbits.

\begin{theo}[Frobenius reciprocity]
   \label{symplectic_frobenius}
   If $X$ is a Hamiltonian $G$-space and $Y$ a Hamiltonian $H$-space, then
   \begin{equation*}
      \Hom_G(X,\operatorname{Ind}^G_HY)= \Hom_H(\operatorname{Res}^G_HX,Y).
   \end{equation*}
\end{theo}

\begin{rems}
Here $\Res^G_HX$ means $X$ regarded as a Hamiltonian $H$-space, and ``$=$'' 
means only a natural bijection as sets. We believe (but haven't proved) that 
both sides are automatically isomorphic as diffeological spaces with 
diffeological $2$-forms as discussed in 
\cite[\S2.5]{Souriau:1985a}, \cite[\S6.38]{Iglesias-Zemmour:2013}, \cite{Karshon:2016}. 
	
Note also that, by the symmetry of \eqref{symplectic_hom}, we may equally write 
Frobenius reciprocity in the form $\Hom_G\left(\IND HGY,Z\right)=
\Hom_H(Y,\Res^G_HZ)$.
\end{rems}

\begin{proof}
For bookkeeping reasons, soon to become clear, rename $G$ also as $K$. 
Consider the spaces $N=X^-\times Y$ with $H$-action $h(x,y)=(h(x), h(y))$, 
and $M=X^-\times T^*K\times Y$ with $K\times H$-action 
$(k,h)(x,\bar p,y)=(k(x),k\bar p h\inv,h(y))$. Their 
equivariant momentum maps are $\psi:N\to\LH^*$,
   \begin{equation}
      \psi(x,y)=\Psi(y)-\Phi(x)\subh
   \end{equation}
   and $\bar\phi\times\bar\psi: M\to\LK^*\times\LH^*$,
   \begin{equation}
      \left\{
      \begin{array}{lll}
         \rlap{$\bar\phi$}\phantom{\bar\psi}(x,\bar p,y) & \!\!= 
         & \!\!\bar p\bar q\inv-\Phi(x)\\[.5ex]
         \bar\psi(x,\bar p,y) & \!\!= & \!\!\Psi(y)- \bar q\inv \bar p\subh
      \end{array}
      \right.
      \rlap{\qquad$(\bar p\in T^*_{\bar q}K)$}
   \end{equation}
where $\Phi$ and $\Psi$ are the equivariant momentum maps of 
$X$ and $Y$, respectively. Defining $r:M\to N$ by 
$r(x,\bar p,y)=(\bar q\inv(x),y)$ now sets us up for a proof using the same 
previous diagram (Fig.~1). Indeed, we have again this time
   \begin{equation}
	   \begin{aligned}
	      \psi((r\circ j_1\circ j)(x,\bar p,y))
	      &=\psi(\bar q\inv(x),y)\\
	      &=\Psi(y)-\Phi(\bar q\inv(x))\subh\\
	      &=\Psi(y)-\bar q\inv\Phi(x)\bar q\subh
	      &&\text{by equivariance}\\
	      &=\Psi(y)-\bar q\inv \bar p\subh
	      &&\text{since }\bar\phi(x,\bar p, y)=0\\
	      &=0\vphantom{\inv}
	      &&\text{since }\bar\psi(x,\bar p, y)=0,
	   \end{aligned}
   \end{equation}
so there is a map $s$ as indicated in Fig.~1. Again, $s$ is onto since 
$(x,y)\mapsto(x,\Phi(x),y)$ provides a right inverse, and $s$ is equivariant
relative to the $K \times  H$-action on 
$(\bar{\phi} \times \bar{\psi})^{-1}(0)$
and the $H$-action on $\psi^{-1}(0)$:
   \begin{equation}
      \begin{aligned}
         s((k,h)(x,\bar p,y))&=r(k(x),k\bar ph\inv,h(y))\\
         &=((k\bar qh\inv)\inv(k(x)), h(y))\\
         &=(h(\bar q\inv(x)),h(y))\\
         &=h(s(x,\bar p,y)).\phantom{\inv}
      \end{aligned}
   \end{equation}
So the fibers of $s$ are again the $K$-orbits and $s$ descends again to a 
bijection $t$ as required and indicated in Fig.~1.
\end{proof}

\section{An Example}
A basic shortcoming in the analogy of 
\eqref{symplectic_frobenius} with representation theory is that it cannot mirror 
cases where more than one representation ``quantizes'' a given Hamiltonian 
$G$-space or $H$-space. To make this assertion precise, we restrict attention to the case where $X$ is a coadjoint orbit of a type I solvable Lie group $G$, endowed with its Kirillov-Kostant-Souriau $2$-form $\omega_X$. In that setting ``quantization'' is well-defined by the theory of Auslander and Kostant \cite{Auslander:1971a} where, we recall:
\begin{itemize}
   \item $X$ has irreducible unitary representations attached to it iff the de Rham cohomology class $[\omega_X]$ belongs to $\mathrm H^2(X,\ZZ)$ (in particular if $\omega_X$ is exact).
   \item If so, and if $G$ is simply connected, then $X$ has as many representations attached to it as there are homomorphisms (a.k.a.~characters) from the fundamental group $\pi_1(X)$ to the circle group $\TT$.
\end{itemize}
This can be summed up by saying that the unitary dual $\widehat G$ is parametrized by \emph{prequantized coadjoint orbits} in the sense of Section 5 below. (One should beware that this can fail beyond the solvable context \cite{Rawnsley:1982,Torasso:1983}: the minimal 
nilpotent coadjoint orbit of 
$\smash{\widetilde{\text{SL}}_3(\RR)}$ has four prequantum bundles 
but only three representations attached to it.) The simplest example where \eqref{symplectic_frobenius} falls short is then as follows. In the solvable group $G'$ of all upper triangular matrices of the form
\begin{equation}
   \label{G'}
   g'=
   \begin{pmatrix}
   \e{\i a}&0&0&b\\
           &1&e&f\\
           & &1&a\\
           & & &1
   \end{pmatrix}
   \qquad
   \begin{aligned}
     a,e,f&\in\RR\\
     b&\in\CC,
   \end{aligned}
\end{equation}
write $G$ for the normal subgroup in which $e=0$ and $H$ for the subgroup of $G$ in 
which $a\in 2\pi\ZZ$. Identify $\LG'^*$ with $\RR\times\CC\times\RR^2$ by 
writing $(p,q,s,t)$ for the value at the identity of the 1-form 
\begin{equation}
	pda+\Re(\bar qdb)-sde-tdf.
\end{equation}
The coadjoint action of $G'$ leaves the hyperplane $t=1$ invariant and acts there by
\begin{equation}
   \label{action}
   g'\begin{pmatrix}p\\q\\s\\1\end{pmatrix}
   =\begin{pmatrix}p+e+\Re(\overline{\i b}q\e{\i a})\\q\e{\i a}\\s+a\\1\end{pmatrix}.
\end{equation}
Likewise, identify $\LG^*$ with triples $(p,q,t)$ and $\LH^*$ with pairs 
$(q,t)$ so that the projections \mbox{$\LG'^*\to\LG^*\to\LH^*$} become 
$(p,q,s,t)\mapsto(p,q,t)\mapsto(q,t)$ and the coadjoint actions are by appropriate restrictions of \eqref{action}. Then the coadjoint orbit 
$X'=G'(\check c)$ of $\check c=(0,1,0,1)$ projects onto the orbit
$X=G(\check c\subg)$ and is its universal covering:
  \begin{equation}
     \label{covering}
     \begin{array}{clll}
        X'
        &\!\!\!=
        &\!\!\!\left\{(p,\e{\i s},s,1) : (p,s)\in\RR^2\right\},
        &\quad\omega_{X'}=dp\wedge ds = d(pds),\\
        \downarrow\\[.5ex]
        X
        &\!\!\!= 
        &\!\!\!\bigl\{(p,q,1) : (p,q)\in\RR\times\TT\bigr\},
        &\quad\omega_X=\smash[t]{dp\wedge \dfrac{dq}{\i q}=d\left(p \dfrac{dq}{\i q}\right)}.
     \end{array}
  \end{equation}
So there is a single representation attached to $X'$, and a circle worth of representations attached to $X$. Moreover, one checks (or finds by \cite{Ziegler:2014} applied to the normal 
subgroup~$H^{\mathrm o}$) that
\begin{equation}
   \label{induced_orbits}
	X=\IND HG{\left\{\check c\subh\right\}},
	\qquad\text{and likewise}\qquad
	X'=\IND{H'}{G'}{\left\{\check c\subhprime\right\}}
\end{equation}
where $H'$ is the normal subgroup of $G'$ in which $a=0$. So symplectic Frobenius reciprocity \eqref{symplectic_frobenius} gives the relation
  \begin{equation}
     \label{frobenius_example}
     \begin{aligned}
        \Hom_G(X,\Res^{G'}_GX')
        &=\Hom_H\bigl(\left\{\check c\subh\right\},\Res^{G'}_HX'\bigr)\\
        &= (X'\to\LH^*)\inv(\check c\subh)/H\\
        &=\{\text{a point}\}
     \end{aligned}
  \end{equation}
between $X$ and the restriction of $X'$. But this fact is of little use for representation 
theory, as it fails to predict into \emph{which} of the representations attached to $X$ the representation attached to $X'$ will split (when restricted to $G$). As one knows, this should be fixed by working instead with \emph{prequantum spaces} in the sense of the next section.

\section{Prequantum $G$-spaces}
Following \cite{Souriau:1970}, we call \emph{prequantum manifold} a manifold 
$\tilde X$ with a contact $1$-form $\varpi$ whose Reeb vector field generates 
a circle group action. We recall that $\varpi$ \emph{contact} means that 
$\Ker(d\varpi)$ is $1$-dimensional and transverse to $\Ker(\varpi)$; its 
\emph{Reeb} vector field, $\i$, on $\tilde X$ is defined by
\begin{equation}
   \i(\tilde x)\in\Ker(d\varpi)
   \qquad\text{and}\qquad
   \varpi(\i(\tilde x)) =1
   \rlap{\qquad$\forall\,\tilde x\in\tilde X.$}
\end{equation}
Then $(\tilde X,d\varpi)$ is a presymplectic manifold whose null leaves are 
the orbits of the circle group $\TT=\mathrm U(1)$ acting on $\tilde X$ and 
$d\varpi$ descends to a symplectic form $\underline\omega$ on the leaf space 
$X=\tilde X/\TT$. If a Lie group $G$ acts on $\tilde X$ and preserves 
$\varpi$, then it commutes with $\TT$ and the equivariant 
momentum map $\Phi:\tilde X\to\LG^*$,
\begin{equation}
   \label{prequantum_moment}
	\<\Phi(\tilde x),Z\>=\varpi(Z(\tilde x)),
\end{equation}
descends to a momentum map $\underline\Phi:X\to\LG^*$, making 
$(X,\underline\omega,\underline\Phi)$ a Hamiltonian $G$-space 
\emph{prequantized} by the \emph{prequantum $G$-space} $(\tilde X,\varpi)$.
 
We do not distinguish between two spaces $\tilde X_1$, $\tilde X_2$ which are 
\emph{isomorphic}, i.e., related by a $G$-equivariant diffeomorphism which 
transforms $\varpi_1$ into~$\varpi_2$. (If several are in play, we may also use 
subscripts like $\varpi_{\tilde X}$, $\i_{\tilde X}$, $\Phi_{\tilde X}$, etc.) 
We recall three basic constructions in the prequantum category:

\begin{dual}
   \label{prequantum_dual}
   (\cite[{}18.47]{Souriau:1970}.) We write $\tilde X^-$ for the $G$-space 
equal to $\tilde X$ but with opposite $1$-form $-\varpi_{\tilde X}$ (and 
consequently opposite Reeb field and $\TT$-action). It prequantizes the dual 
$G$-space $(X^-,-\underline\omega,-\underline\Phi)$.
\end{dual}

\begin{boxp}
   \label{prequantum_product}
   (\cite[{}18.52]{Souriau:1970}.) If $\tilde X_1$ and $\tilde X_2$ are prequantum $G$\nobreakdash-spaces, then $\tilde X_1\times\tilde X_2$ (with diagonal $G$-action) is a $\TT^2$-space in which the action of the 
anti-diagonal $\Delta=\{(z\inv,z) : z\in\TT\}$ has as its orbits the characteristic leaves of the $1$-form $\varpi_1+\varpi_2$. Hence this descends to the quotient $\tilde X_1\boxtimes\tilde X_2:=(\tilde X_1\times\tilde X_2)/\Delta$ as a $1$-form making it a prequantization of the symplectic product $X_1\times X_2$. In view of \eqref{prequantum_dual}, the $\Delta$-action on $\tilde X_1^{-}\times\tilde X_2^{\phantom-}$ is $z(\tilde x_1,\tilde x_2) = (z(\tilde x_1),z(\tilde x_2))$.
\end{boxp}

\begin{redu}
   \label{prequantum_reduction}
   (\cite[Thm 2]{Loose:2001}.) Assume $G$ acts freely and properly on 
$\tilde X$, and consider the level $L:=\Phi\inv(0)$. By the very definition \eqref{prequantum_moment} 
of $\Phi$ and its being a momentum map, we~have $\LG(\tilde x)
\subset\smash{\Ker(\varpi_{|L})\cap\Ker(d\varpi_{|L})}$. Since 
$\smash{\varpi_{|L}}$ is also $G$-invariant, it follows (see 
\cite[{}5.21]{Souriau:1970}) that it descends to a contact $1$-form on the 
quotient $\tilde X/\!\!/G:=\Phi\inv(0)/G$. This prequantizes the symplectic 
reduction $X/\!\!/G=\underline\Phi\inv(0)/G$.
\end{redu}

\section{Prequantum Induction}

Given a closed subgroup $H\subset G$ and a prequantum $H$-space 
$(\tilde Y,\varpi_{\tilde Y})$ whose momentum map \eqref{prequantum_moment} 
we denote~$\Psi$, we propose to construct an \emph{induced} prequantum $G$-space 
$\IND HG{\tilde Y}$ as follows. Consider the prequantum $(G\times H)$-space 
$\smash{\tilde N}=T^*G\times\smash{\tilde Y}$ with $1$-form 
$\varpi_{T^*G}+\varpi_{\tilde Y}$ and action 
$(g,h)(p,\tilde y)=(gph\inv, h(\tilde y))$. This action has the equivariant 
momentum map $\phi\times\psi:\smash{\tilde N}\to\LG^*\times\LH^*$,
\begin{equation}
   \left\{
   \begin{array}{lll}
      \rlap{$\phi$}\phantom{\psi}(p,\tilde y) & \!\!= & \!\!pq\inv\\[.5ex]
      \psi(p,\tilde y) & \!\!= & \!\!\Psi(\tilde y)- q\inv p\subh
   \end{array}
   \right.
   \rlap{\qquad$(p\in T^*_qG)$.}
\end{equation}
The same arguments as with \eqref{induced_manifold}, then, show that
\begin{equation}
   \label{prequantum_induction}
   \IND HG{\tilde Y}:=\tilde N/\!\!/H = \psi\inv(0)/H
\end{equation}
(prequantum reduction \eqref{prequantum_reduction}) is naturally a prequantum 
$G$-space which prequantizes the symplectically induced manifold 
\eqref{induced_manifold}.

\section{Prequantum Induction in Stages}

\begin{theo}[Stages]
   If $H \subset K \subset G$ are closed subgroups
   of the Lie group $G$\textup, then
   \begin{equation*}
      \IND KG{\IND HK{\tilde Y}}=\IND HG{\tilde Y}.   
   \end{equation*}
\end{theo}

\begin{proof}
The proof is \emph{mutatis mutandis} the same as for \eqref{symplectic_stages}, 
only simpler. We just switch to working with restrictions and push-forwards of 
the $1$-form $\varpi(\delta p,\delta\bar p,\delta\tilde y)=\<p,\delta q\>+\<\bar p,\delta\bar q\>+\smash{\varpi_{\tilde Y}(\delta\tilde y)}$ on 
$\tilde M=T^*G\times T^*K\times \tilde Y$ instead of the $2$-form 
$\underline\omega$ on $M$.
\end{proof}

\section{Prequantum Frobenius Reciprocity}

The three constructions (\ref{prequantum_dual}--\ref{prequantum_reduction}) put together furnish us with a notion of the 
\textit{intertwiner space} of two prequantum $G$-spaces,
\begin{equation}
   \label{prequantum_hom}
   \Hom_G(\tilde X_1, \tilde X_2):=
   (\tilde X_1^-\boxtimes\tilde X_2^{\vphantom-})/\!\!/G.
\end{equation}
Freeness and properness of the last $G$-action are not 
assumed and we again regard \eqref{prequantum_hom} as just a set.

\begin{theo}[Frobenius reciprocity]
\label{contact_frobenius}
If $\tilde X$ is a prequantum $G$-space and $\tilde Y$ a prequantum $H$-space, 
then
   \begin{equation*}
      \Hom_G(\tilde X,\IND HG{\tilde Y})=\Hom_H(\Res^G_H\tilde X,\tilde Y).
   \end{equation*}
\end{theo}

\begin{proof}
   With $\Delta$ as in \eqref{prequantum_product}, define $\tilde{\tilde r}$ in the following commutative diagram by $\tilde{\tilde r}(\tilde x,p,\tilde y)=(q\inv(\tilde x),\tilde y)$, where $p\in T^*_qG$:
   \begin{equation}
      \label{diagram}
      \begin{tikzcd}[arrows={line width=1.2*rule_thickness}]
         \tilde{\tilde M}:=\tilde X^-\times T^*G\times\tilde Y
         \rar[rrr,"\tilde{\tilde r}"]
         \dar[dd,"\mod\Delta",swap]
         & & &
         \tilde{\tilde N}:=\tilde X^-\times\tilde Y
         \dar[dd,"\mod\Delta" inner sep=-1ex]
         \\
         \\
         \tilde M:=\tilde X^-\boxtimes T^*G\times\tilde Y
         \rar[rrr,"\tilde r",dashed]
         \dar[dd,"\mod(\TT^2/\Delta)",swap]
         & & &
         \tilde N:=\tilde X^-\boxtimes\tilde Y
         \dar[dd,"\mod(\TT^2/\Delta)" inner sep=-1ex]
         \\
         \\
         M:= X^-\times T^*G\times Y
         \rar[rrr,"r",dashed]
         & & &
         N:= X^-\times Y.
      \end{tikzcd}
   \end{equation}
   Then $\tilde{\tilde r}$ descends, as indicated, to a map $\tilde r$ and a map $r$ which is the one in our proof of \eqref{symplectic_frobenius}. Now each floor of this diagram supports a horizontal copy of Fig.~1 giving rise to the appropriate tilded versions of $s$ and $t$; a straightforward diagram chase checks that $\tilde t:(\tilde M/\!\!/H)/\!\!/G\to\tilde N/\!\!/H$ is the required bijection.
\end{proof}

\section{An Example (Reprise)}

\noindent
Recall the coadjoint orbits $X\cong\RR\times\TT$ and $X'\cong\RR^2$ of 
\eqref{covering}. Referring to 
\cite[{}18.117, 18.133, 18.134]{Souriau:1970} as well as \cite[Thm 5.1.1]{Kostant:1970} and performing direct verifications, one finds:
\begin{itemize}
   \item There is a unique prequantum $G'$-space prequantizing $X'$, namely
   \begin{equation}
      \tilde X' = X'\times\TT\ni(p,s,z)
      \qquad\text{with}\qquad
      \varpi = pds+\frac{dz}{\i z}
   \end{equation}
   with $G'$-action (uniquely lifted from \eqref{action} to preserve $\varpi$)
   \begin{equation}
      \label{prequantum_action}
      g'\begin{pmatrix}p\\s\\z\end{pmatrix}
      =\begin{pmatrix}
         p + e + \Re(\overline{\i b}\e{\i(s+a)})\\
         s+a\\
         z\e{\i\left[\Re\left(\overline b\e{\i(s+a)}\right)-se-f\right]}
      \end{pmatrix}.
   \end{equation}
   \item There is a circle worth of inequivalent prequantum $G$-spaces over $X$, namely all 
   \begin{equation}
      \tilde X_\lambda = X\times\TT\ni(p,q,z)
      \qquad\text{with}\qquad
      \varpi_\lambda = (p+\lambda)\frac{dq}{\i q}+\frac{dz}{\i z}
   \end{equation}
   where $\lambda_1, \lambda_2\in\RR$ give equivalent prequantizations iff they differ by an integer (cf.~\cite{Aharonov:1959, Kastrup:2006}), and the $G$-action on $\tilde X_\lambda$ (uniquely lifted from that on $X$ to preserve $\varpi_\lambda$) reads
   \begin{equation}
      g\begin{pmatrix}p\\q\\z\end{pmatrix}
      =\begin{pmatrix}
         p + \Re(\overline{\i b}q\e{\i a})\\
         q\e{\i a}\\
         z\e{\i\left[\Re\left(\overline bq\e{\i a}\right)-\lambda a-f\right]}
      \end{pmatrix}.
   \end{equation}
\end{itemize}
Moreover, from \eqref{induced_orbits} and the sentence following \eqref{prequantum_induction} one deduces without trouble that
\begin{equation}
   \tilde X_\lambda=\IND HG{\TT_\lambda}
   \qquad\text{and}\qquad
	\tilde X'=\IND{H'}{G'}\TT'
\end{equation}
where $\TT_\lambda$ is the unit circle on which $H$ acts by the character $\chi_\lambda(h)=\e{-\i\lambda a}\e{\i[\Re(b) - f]}$, and $\TT'$ is the unit circle on which $H'$ acts by the character $\chi'(h')=\e{\i[\Re(b) - f]}$ (notation \eqref{G'}). Consequently \eqref{contact_frobenius} gives 
\begin{equation}
   \label{last_Frobenius}
   \begin{aligned}
    \Hom_G(\tilde X_\lambda,\Res^{G'}_G\tilde X')&=
       \Hom_H(\TT_\lambda,\Res^{G'}_H\tilde X')\\
       &= \text{a single circle},
   \end{aligned}
\end{equation} 
as one easily computes from (\ref{prequantum_hom}, \ref{prequantum_action}). This replaces \eqref{frobenius_example} and ``predicts'' that once 
restricted to $G$, the irreducible representation $\IND{H'}{G'}\chi'$ (attached to $\tilde X'$ by Auslander-Kostant \cite{Auslander:1971a}) splits into the direct integral over $\lambda\in\RR/\ZZ$ 
of the irreducible representations $\IND HG{\chi_\lambda}$ (attached to $\tilde X_\lambda$) with 
multiplicity $1$; this prediction is correct and can be checked directly.

\let\i\dotlessi
\let\l\polishl
\let\o\norwegiano
\let\u\russianbreve

\section*{Funding} T.~R. was partially supported by the 
National Natural Science Foundation of China [grant number 11871334] and by 
the Swiss National Science Foundation [NCCR SwissMAP].

\end{document}